\documentclass[10pt,journal]{IEEEtran}

\usepackage{graphicx}
\usepackage{amsmath,cite,mathrsfs}
\usepackage{amsfonts}

\begin{document}

\newtheorem{algorithm}{Algorithm}
\newtheorem{assumption}{Assumption}
\newtheorem{corollary}{Corollary}
\newtheorem{definition}{Definition}
\newtheorem{example}{Example}
\newtheorem{lemma}{Lemma}
\newtheorem{remark}{Remark}
\newtheorem{theorem}{Theorem}
\renewcommand\abstractname{Abstract}
\renewcommand\figurename{Figure}
\renewcommand\tablename{ Table}
\renewcommand\refname{References}

\title{State Estimation for Genetic Regulatory Networks with Time-Varying Delays and Reaction-Diffusion Terms}

\author{Yuanyuan Han, Xian~Zhang,
~\IEEEmembership{Member,~IEEE}, Ligang~Wu,~\IEEEmembership{Senior~Member,~IEEE} and Yantao Wang
\thanks{Y. Han, X. Zhang (Corresponding author) and Y. Wang are with the School
of Mathematical Science, Heilongjiang University, Harbin 150080, China,
Email: xianzhang@ieee.org.}
\thanks{L. Wu is with the Space Control and Inertial Technology Research Center,
Harbin Institute of Technology, Harbin, 150001, China, Email: ligangwu@hit.edu.cn.}
\thanks{This work was supported in part by the National Natural Science Foundation of China (11371006, 61174126 and 61222301),
the National Natural Science Foundation of Heilongjiang Province (F201326,A201416),
the Fund of Heilongjiang Education Committee (12541603),
the Fundamental Research Funds for the Central Universities (HIT.BRETIV.201303),
and the Heilongjiang University Innovation Fund for Graduates (YJSCX2015-033HLJU).
}
}


\maketitle

\begin{abstract}
This paper is concerned with the state estimation problem for genetic regulatory networks with time-varying delays and reaction-diffusion terms under Dirichlet boundary conditions. It is assumed that the nonlinear regulation function is of the Hill form. The purpose of this paper  is to design a state observer to estimate the concentrations of mRNA and protein through available measurement outputs.
By introducing new integral terms in a novel Lyapunov--Krasovskii functional and employing Wirtinger-based integral inequality, Wirtinger's inequality, Green's identity, convex combination approach, and reciprocally convex combination approach, an asymptotic stability criterion of the error system is established in terms of linear matrix inequalities
(LMIs). The obtained stability criterion depends on the upper bounds of the delays and their derivatives. It should be highlight that if the set of LMIs are feasible, the desired observer exists and can be determined. Finally, two numerical examples are presented to illustrate the effectiveness of the proposed designed scheme.

\end{abstract}

\IEEEpeerreviewmaketitle

\section{Introduction}\label{intro}
\IEEEPARstart{I}{}n
the last decades, due to the increasing progress in genome sequencing and gene recognition, genetic regulatory networks (GRNs) have become a significant area in biological and biomedical sciences. However, there still exists large gap between the genome sequencing and the understanding of gene functions which have become challenge problems in system biology. A great amount of experimental results show that mathematical modeling of GRNs can be a powerful tool for researching the gene regulation process and discovering
complex structure of a biological organism \cite{JCB2-2002-67,PSB-1999-17,PSB-1998-77}.
Generally, there are two basic models for GRNs: Boolean model (discrete-time model) \cite{Cao.Ren(2008),AAA-2014-257971} and differential equation model (continuous-time model) \cite{MPE-2014-768483,IEEE-T-CBB-2015-398,Zhang.Wu.Zou-IEEE-T-CBB}. Differential equation model describes the change rates of the concentrations of mRNAs and proteins. Furthermore, differential equation model has been most frequently utilized since it can more precisely describe the whole network and made it possible to understand the dynamic behavior of whole network in detail.

In biological systems, particularly in GRNs, stability is the most significant and essential dynamical behaviors \cite{IEEE-T-NNLS-2013-1957,IEEE-T-NNLS-2013-1177,NN-2014-165,IEEE-T-NNLS-2012-293}.
It is related to not only the structure and function of an organism but also the strength and characteristics of the external disturbances. In addition, as it is well known, time delays caused by the slow processes of transcription and translation in real GRNs. It has been well shown from present research results that time delay may lead to instability, bifurcation or oscillation for systems \cite{IEEE-T-2005-217,Li.Gao.Shi(2010),IEEE-T-IE-2012-2732,N-2015-199}. However, mathematical modeling of GRNs without introducing delays will lead to wrong predictions of the concentrations of mRNAs and proteins. Therefore, the problem of stability analysis for biological systems with time delays has stirred  increasing research interests and a great deal of excellent results has been reported in literature in recent decades (see, e.g., \cite{IEEE-T-CBB-2015-398,Zhang.Wu.Zou-IEEE-T-CBB,CSSP-2014-371,IEEE-T-SMC-B-2009-467,N-2013,NCA-2013,IEEE-T-AC-2013-475,IEEE-T-NN-2008-1299}).

In some mathematical modeling, it is  implicitly assumed that the genetic regulatory systems are spatially homogeneous, namely, the concentrations of mRNA and protein are homogenous in space at all times. However, there are some situations in which these assumptions are not reasonable. For instance, it might be necessary to consider the diffusion of regulatory proteins from one compartment to another \cite{JCB2-2002-67,JMB-1985-313,britton1986reaction,Han.Zhang.Wang-CSSP}.
In this situation, the general functional differential equation model can not precisely describe genetic regulatory process more or less. Hence, it is imperative to introduce reaction-diffusion terms in mathematical modeling of GRNs.
 To the best of authors' knowledge, the delayed GRNs with reaction-diffusion terms are only studied in \cite{Zhou.Xu.Shen(2011),NCA-2011-507,001,Han.Zhang.Wang-CSSP}. Ma \textit{et al}. \cite{NCA-2011-507} introduced reaction-diffusion terms to GRNs for the first time and established delay-dependent asymptotic stability criteria. Based on the Lyapunov functional method, Zhou, Xu and Shen \cite{Zhou.Xu.Shen(2011)} investigated finite-time robust stochastic stability criteria for uncertain GRNs with time-varying delays and reaction-diffusion terms. Han and Zhang \cite{001,Han.Zhang.Wang-CSSP} gradually improved Ma \textit{et al}.'s results by introducing novel Lyapunov--Krasovskii functional and employing Jens\-en's inequality, Wirtinger's inequality, Green's identity, convex combination approach and reciprocally convex combination (RCC) approach.

In complex biological networks such as neural networks and GRNs, it is often the case that only partial information about the states of the nodes is available in the network outputs. In order to understand biological networks better, it is indispensable to estimate the states of the nodes through available measurements. Hence, the problem of state estimation for biological networks has been one of the investigated dynamical behaviors in recent years \cite{CNSNS-2011-4060,JFI-2013-966,N-2015-168,IEEE-T-NN-2005-279--284,AAA-2014-257971,OCAM-2012-590}.
%
%
However, to the best of the authors' knowledge, there is still no any published result on the state estimation problem for GRNs with time-varying delays and reaction-diffusion terms, which  arouses our research interests.

Motivated by above discussion, we aim to investigate the state estimation problem for GRNs with time-varying delays and reaction-diffusion terms. By introducing new integral terms into a novel Lyapunov--Krasovskii functional and employing Wirtinger-based integral inequality, Wirtinger's inequality, Green's  identity, convex combination approach and RCC approach, an asymptotic stability criterion of the error system is established in terms of LMIs. Thereby, a state observer is designed, and the observer gain matrices are described in terms of the solution to a set of LMIs.

The rest of the paper is organized as follows: the problem is formulated and some preliminaries
are given in Section 2; in Section 3, an asymptotic stability criterion for the error system is established, and an approach to design state observer is proposed; two numerical
examples are provided in Section 4; and finally, we conclude this paper in Section 5.

\textbf{Notation}
We now set some standard notations, which will be used in the rest of the paper.
$I$ is the identity matrix
with appropriate dimension, $A^{T}$ represents the transpose of the matrix $A$.
For real symmetric matrices $X$ and $Y$, $X>Y(X\geq Y)$ means that $X-Y$ is positive definite (positive semi-definite).
$\Omega$ is a compact set in the vector space
$\mathbb{R}^{n}$ with smooth boundary $\partial\Omega$. Let $C^{k}(X ,Y)$ be the Banach space of functions which map $X$ into $Y$ and have continuous $k$-order derivatives.
For a positive integer $n$, let $\langle n\rangle$ be the set $\{1,2,\ldots,n\}$.

\section{Model description and preliminaries}\label{2014-3-15-1511-section-1}

This paper considers the following GRNs with time-varying delays and react\-i\-on-diffusion terms \cite{Han.Zhang.Wang-CSSP}:
\begin{equation}\label{equation-1}\setlength{\arraycolsep}{0.5pt}
\left\{\begin{array}{rl}
\displaystyle\frac{\partial\tilde{m}_i(t,x)}{\partial t}=&\sum^{l}_{k=1}\displaystyle\frac{\partial}{\partial x_{k}}\left(D_{ik}\displaystyle\frac{\partial\tilde{m}_{i}(t,x)}{\partial x_{k}}\right)-a_i\tilde{m}_i(t,x)\\&+\sum^n_{j=1}w_{ij}g_j(\tilde{p}_j(t-\sigma(t),x))+q_i, \\
\displaystyle\frac{\partial\tilde{p}_i(t,x)}{\partial t}=&\sum^{l}_{k=1}\displaystyle\frac{\partial}{\partial x_{k}}\left(D^{*}_{ik}\displaystyle\frac{\partial\tilde{p}_{i}(t,x)}{\partial x_{k}}\right)-c_i\tilde{p}_i(t,x)\\&+b_i\tilde{m}_i(t-\tau(t),x),
\end{array}\right.
\end{equation}
where $i\in\langle n \rangle$, 
$x=\mathrm{col}(x_{1},x_{2},\dots,x_{l})\in\Omega\subset \mathbb{R}^{l}$,
     $\Omega=\{x\bigl||x_{k}|\leq L_{k},k\in\langle l\rangle\}$, $L_{k}$ is a given constant; $D_{ik}>0$ and $D^{*}_{ik}>0$ denote the diffusion rate matrices; $\tilde{m}_{i}(t,x)$ and $\tilde{p}_{i}(t,x)$ are the concentrations of mRNA and
protein of the $i$th node, respectively; $a_i$ and $c_i$ are degradation rates of the mRNA and protein, respectively; $b_i$ represents the translation rate; $W:=[w_{ij}]\in \mathbb{R}^{n\times n}$ is the coupling matrix of the genetic networks, which is defined as follows:
$$
w_{ij}=\left\{ \begin{array}{ll}
\gamma_{ij},  &\hbox{if $j$ is an activator of gene $i$, }\\
0           , &\hbox{if there is no link from gene $j$ to $i$},\\
-\gamma_{ij}, &\hbox{if $j$ is a repressor of  gene $i$, }
\end{array}\right.
$$
here $\gamma_{ij}$ is the dimensionless transcriptional rate of transcription factor $j$ to gene $i$; $g_{j}$ represents the feedback regulation function of protein on transcription, which is the monotonic function in Hill form, i.e., $g_{j}(s)=\frac{s^{H}}{1+s^{H}}$, where $H$ is the Hill coefficient; $q_{i}=\Sigma_{j\in I_{i}}\gamma_{ij}$, $I_{i}$ is the set of all the nodes which are repressors of gene $i$; $\sigma(t)$ and $\tau(t)$ are time-varying delays satisfying
\begin{equation}\label{equation-2}
\begin{array}{ll}
0\leq\tau(t)\leq\overline{\tau}, \ \dot{\tau}(t)\leq\mu_{1}, \\
0\leq\sigma(t)\leq\overline{\sigma}, \ \dot{\sigma}(t)\leq\mu_{2},
\end{array}
\end{equation}
where $\overline{\tau}$, $\overline{\sigma}$, $\mu_{1}$ and $\mu_{2}$ are non-negative real numbers.

%
The initial conditions associated with GRN (\ref{equation-1}) are given as follows:
$$
\begin{array}{lll}
\tilde{m}_{i}(s,x)=\phi_{i}(s,x), &x\in\Omega, &s\in[-d,0], i\in\langle n\rangle,\\
\tilde{p}_{i}(s,x)=\phi^{*}_{i}(s,x), &x\in\Omega, &s\in[-d,0], i\in\langle n\rangle,\\
\end{array}
$$
where $d=\mathrm{\mathrm{max}}\{\overline{\sigma}, \overline{\tau}\}$, and $\phi_{i}(s,x)$, $\phi^{*}_{i}(s,x)\in C^{1}([-d,0]\times\Omega , \mathbb{R})$.

In this paper, the following type of boundary conditions (Dirichlet boundary conditions) is considered:
$$
\begin{array}{lll}
\tilde{m}_{i}(t,x)=0, &x\in\partial\Omega, &t\in[-d,+\infty),\\
\tilde{p}_{i}(t,x)=0, &x\in\partial\Omega, &t\in[-d,+\infty).
\end{array}
$$

Now, we assume that
$$m^{*}(x):=\mathrm{col}(m^{*}_{1}(x),m^{*}_{2}(x),\dots,m^{*}_{n}(x))$$ and $$p^{*}(x):=\mathrm{col}(p^{*}_{1}(x),p^{*}_{2}(x),\dots,p_{n}^{*}(x))$$
are the unique equilibrium solution of GRN (\ref{equation-1}), that is,
$$
\left\{\begin{array}{rl}
0=&\sum^{l}_{k=1}\displaystyle\frac{\partial}{\partial x_{k}}\left(D_{k}\displaystyle\frac{\partial m^{*}_{i}(x)}{\partial x_{k}}\right)-a_i m^{*}_{i}(x)\\
&+\sum^n_{j=1}w_{ij}g_j(p^{*}_j(x))+q_i,\\
0=&\sum^{l}_{k=1}\displaystyle\frac{\partial}{\partial x_{k}}\left(D^{*}_{k}\displaystyle\frac{\partial p^{*}_{i}(x)}{\partial x_{k}}\right)-c_i p^{*}_{i}(x)+b_i m^{*}_i(x)
\end{array}\right.
$$
for $i\in\langle n\rangle$. Obviously, the transformations, $\bar{m}_{i}=\tilde{m}_{i}-m_{i}^{*}$ and $\bar{p}_{i}=\tilde{p}_{i}-p_{i}^{*}$, $i\in\langle n\rangle,$ transform GRN (\ref{equation-1}) into the following matrix form:
\begin{equation}\label{equation-4}
\left\{\begin{array}{rl}
\displaystyle\frac{\partial \bar{m}(t,x)}{\partial t}=&\sum^{l}_{k=1}\displaystyle\frac{\partial}{\partial x_{k}}\left(D_{k}\displaystyle\frac{\partial \bar{m}(t,x)}{\partial x_{k}}\right)\\
&-A\bar{m}(t,x)+Wf(\bar{p}(t-\sigma(t),x)),\\
\displaystyle\frac{\partial \bar{p}(t,x)}{\partial t}=&\sum^{l}_{k=1}\displaystyle\frac{\partial}{\partial x_{k}}\left(D^{*}_{k}\displaystyle\frac{\partial \bar{p}(t,x)}{\partial x_{k}}\right)\\
&-C\bar{p}(t,x)+B\bar{m}(t-\tau(t),x),
\end{array}\right.
\end{equation}where
$$A=\mathrm{diag}(a_{1},a_{2},\ldots,a_{n}), C=\mathrm{diag}(c_{1},c_{2},\ldots,c_{n}), $$
$$B=\mathrm{diag}(b_{1},b_{2},\ldots,b_{n}),$$
$$D_{k}=\mathrm{diag}(D_{1k},D_{2k},\ldots,D_{nk}),$$
$$ D^{*}_{k}=\mathrm{diag}(D^{*}_{1k},D^{*}_{2k},\ldots,D^{*}_{nk}),$$
$$\bar{m}(t,x)=\mathrm{col}(\bar{m}_{1}(t,x),\bar{m}_{2}(t,x),\dots,\bar{m}_{n}(t,x)),$$
$$\bar{p}(t,x)=\mathrm{col}(\bar{p}_{1}(t,x),\bar{p}_{2}(t,x),\dots,\bar{p}_{n}(t,x)), $$
$$
\begin{array}{rl}
&f(\bar{p}(t-\sigma(t),x))\\
=&\mathrm{col}(f_{1}(\bar{p}_{1}(t-\sigma(t),x)),\cdots,f_{n}(\bar{p}_{n}(t-\sigma(t),x))),
\end{array}
$$
$$
f_{i}(\bar{p}_{i}(t-\sigma(t),x))=g_{i}(\bar{p}_i(t-\sigma(t),x)+p^{*}_{i})-g_{i}(p^{*}_{i}), \ i\in\langle n\rangle.
$$

Because of the complexity of GRN (\ref{equation-4}), it is normally of the case that
only partial information about the states of the nodes is available in the network outputs. In order to obtain the true state of (\ref{equation-4}), it becomes necessary to estimate the states of the nodes through network measurements. The available measurements are given as follows:
\begin{equation}\label{equation-6}
\left\{\begin{array}{l}
z_{m}(t,x)=M\bar{m}(t,x),\\
z_{p}(t,x)=N\bar{p}(t,x),
\end{array}\right.
\end{equation}
where $z_{m}(t,x)$ and $z_{p}(t,x)$ are the actual measurement outputs, and $M$ and $N$ are known constant matrices with appropriate
dimensions.

To estimate the states of GRN (\ref{equation-4}) through available measurement
outputs in (\ref{equation-6}), we construct the following state observer:
\begin{equation}\label{equation-7}
\left\{\begin{array}{rl}
\hspace*{-3mm}\displaystyle\frac{\partial \hat{m}(t,x)}{\partial t}=
&\hspace*{-3mm}\sum^{l}_{k=1}\displaystyle\frac{\partial}{\partial x_{k}}\left(D_{k}\displaystyle\frac{\partial \hat{m}(t,x)}{\partial x_{k}}\right)-A\hat{m}(t,x)\\
&\hspace*{-3mm}+Wf(\hat{p}(t-\sigma(t),x))\\
&\hspace*{-3mm}+K_{1}[z_{m}(t,x)-M\hat{m}(t,x)],\\
\hspace*{-3mm}\displaystyle\frac{\partial \hat{p}(t,x)}{\partial t}=
&\hspace*{-3mm}\sum^{l}_{k=1}\displaystyle\frac{\partial}{\partial x_{k}}\left(D^{*}_{k}\displaystyle\frac{\partial \hat{p}(t,x)}{\partial x_{k}}\right)-C\hat{p}(t,x)\\
&\hspace*{-3mm}+B\hat{m}(t-\tau(t),x)+K_{2}[z_{p}(t,x)-N\hat{p}(t,x)],
\end{array}\right.
\end{equation}
where $\hat{m}(t,x)$ and $\hat{p}(t,x)$ are the estimations of $m(t,x)$ and $p(t,x)$, respectively, and $K_{1}$ and $K_{2}$ are the observer gain matrices to be designed later.

The initial conditions for the state observer (\ref{equation-7}) are assumed to be $(\hat{m}_{i}(t,x),\hat{p}_{i}(t,x))=(\phi_{i}(s,x),\phi^{*}_{i}(s,x))$.

Our aim is to find suitable observer gains $K_{1}$ and $K_{2}$, so that $\hat{m}(t,x)$ and $\hat{p}(t,x)$, respectively, approach to $m(t,x)$ and $p(t,x)$ as $t\rightarrow+\infty$.
Let the error state vectors be $m(t,x)=\bar{m}(t,x)-\hat{m}(t,x)$ and $p(t,x)=\bar{p}(t,x)-\hat{p}(t,x)$. Then it follows from (\ref{equation-4}), (\ref{equation-6}) and (\ref{equation-7}) that
\begin{equation}\label{equation-8}
\left\{\begin{array}{rl}
\hspace*{-3mm}\displaystyle\frac{\partial m(t,x)}{\partial t}=
&\hspace*{-3mm}\sum^{l}_{k=1}\displaystyle\frac{\partial}{\partial x_{k}}\left(D_{k}\displaystyle\frac{\partial m(t,x)}{\partial x_{k}}\right)\\
&\hspace*{-3mm}-(A+K_{1}M)m(t,x)+W\bar{f}(p(t-\sigma(t),x)),\\
\hspace*{-3mm}\displaystyle\frac{\partial p(t,x)}{\partial t}=&\hspace*{-3mm}\sum^{l}_{k=1}\displaystyle\frac{\partial}{\partial x_{k}}\left(D^{*}_{k}\displaystyle\frac{\partial p(t,x)}{\partial x_{k}}\right)\\
&\hspace*{-3mm}-(C+K_{2}N)p(t,x)+Bm(t-\tau(t),x),
\end{array}\right.
\end{equation}
where
$$
\label{equation-8-1}\bar{f}(p(t-\sigma(t),x))=f(\bar{p}(t-\sigma(t),x))-f(\hat{p}(t-\sigma(t),x)).
$$

From the relationship among $\bar{f}_i$, $f_{i}$ and $g_{i}$, one can easily obtain that
$$\bar{f}_{i}(0)=0, 0\leq\frac{\bar{f}_{i}(y)}{y}\leq\xi_{i}, \forall y\in \mathbb{R}, y \neq0, i\in\langle n\rangle, $$ namely,
 \begin{equation}\label{equation-5}
\bar{f}(0)=0, \bar{f}^{T}(z)(\bar{f}(z)-Kz)\leq0,\forall z\in \mathbb{R}^{n},
\end{equation}
where $K=\mathrm{diag}(\xi_{1},\xi_{2},\ldots,\xi_{n})>0$.

In this paper, we assume that error system (\ref{equation-8}) satisfies Dirichlet boundary conditions:
$$
\begin{array}{lll}
m_{i}(t,x)=0, &x\in\partial\Omega, &t\in[-d,+\infty),\\
p_{i}(t,x)=0, &x\in\partial\Omega, &t\in[-d,+\infty).
\end{array}
$$

We introduce the following lemmas which play key roles in obtaining the main results of this paper.
\begin{lemma}[Jensen's Inequality]\cite{Gu(2000),A-2010-466}\label{lemma3}
For any constant matrix $M^{T}=M>0$ of appropriate dimension, any scalars $a$ and $b$ with $a<b$, and a vector function $w:[a,b]\rightarrow \mathbb{R}^{n}$ such that the integrals concerned are
well defined, then the following inequality holds:
$$\left(\int^{b}_{a}\hspace*{-2mm}w(s)\mathrm{d}s\right)^{T}M\left(\int^{b}_{a}\hspace*{-2mm}w(s)\mathrm{d}s\right)\leq(b-a)\int^{b}_{a}\hspace*{-2mm}w^{T}(s)Mw(s)\mathrm{d}s,$$
$$\renewcommand{\arraystretch}{2}\begin{array}{rl}
&\left(\int^{b}_{a}\int^{b}_{\theta}w(s)\mathrm{d}s\mathrm{d}\theta \right)^{T}M\left(\int^{b}_{a}\int^{b}_{\theta}w(s)\mathrm{d}s\mathrm{d}\theta\right)\\ \leq&\frac{(b-a)^{2}}{2}\int^{b}_{a}\int^{b}_{\theta}w^{T}(s)Mw(s)\mathrm{d}s\mathrm{d}\theta.$$
\end{array}$$
\end{lemma}

\begin{lemma}[Wirtinger-based Integral Inequality]\label{2015-lemma-1}\cite{A-2013-2860}
For given a symmetric positive definite matrix $Q\in \mathbb{R}^{n\times n}$, and a differentiable function $\omega : [a,b] \rightarrow \mathbb{R}^{n}$, the following inequality holds:
$$
\int_a^b\dot{w}^\mathrm{T}(u)Q\dot{w}(u){\mathrm{d}}u\nonumber \ge   \frac{1}{b-a}
\left[\begin{array}{rl}
\Omega_0\\
\Omega_1
\end{array}\right]^\mathrm{T}\tilde{Q}
\left[\begin{array}{rl}
\Omega_0\\
\Omega_1
\end{array}\right],
$$
where $\tilde{Q} = \mathrm{diag}(Q,3Q)$, $\Omega_0 = w(b)-w(a)$ and
\begin{eqnarray*}
\Omega_1 = w(b)+w(a)-\frac{2}{b-a}\int_a^b w(u){\mathrm{d}}u. \nonumber
\end{eqnarray*}
\end{lemma}

\begin{lemma}[Wirtinger's Inequality]\cite{kammler2007first}\label{lemma2}
Assume that the function $f\in C^{1}([a,b],\mathbb{R}^{n})$
satisfies $f(a)=f(b)=0$. Then
$$\int^{b}_{a}f^{2}(v)\mathrm{d}v\leq\frac{(b-a)^{2}}{\pi^{2}}\int^{b}_{a}[f'(v)]^{2}\mathrm{d}v.$$
\end{lemma}

\begin{lemma}\label{lemma5}\cite{Han.Zhang.Wang-CSSP}
Let $N_{1}>0$ and $N_{2}>0$ be a pair of diagonal matrices. Then the states of (\ref{equation-8}) satisfy
$$\begin{array}{rl}
&\int_{\Omega}\frac{\partial m^{T}(s,x)}{\partial t}N_{1}\sum^{l}_{k=1}\frac{\partial}{\partial x_{k}}\left(D_{k}\frac{\partial m(t,x)}{\partial x_{k}}\right)\mathrm{d}x
\\=&\int_{\Omega}m^{T}(t,x)N_{1}\sum^{l}_{k=1}\frac{\partial}{\partial x_{k}}\left[D_{k}\frac{\partial}{\partial x_{k}}\left(\frac{\partial m(t,x)}{\partial t}\right)\right]\mathrm{d}x,
\end{array}
$$
$$\begin{array}{rl}
&\int_{\Omega}\frac{\partial p^{T}(s,x)}{\partial t}N_{2}\sum^{l}_{k=1}\frac{\partial}{\partial x_{k}}\left(D^{*}_{k}\frac{\partial p(t,x)}{\partial x_{k}}\right)\mathrm{d}x
\\=&\int_{\Omega}p^{T}(t,x)N_{2}\sum^{l}_{k=1}\frac{\partial}{\partial x_{k}}\left[D^{*}_{k}\frac{\partial}{\partial x_{k}}\left(\frac{\partial p(t,x)}{\partial t}\right)\right]\mathrm{d}x.
\end{array}
$$
\end{lemma}

\begin{lemma}[RCC Lemma]\label{lemma6}
\cite{A-2011-235} Let $f_1,f_2,\dots,f_N:\mathbb{D}\rightarrow \mathbb{R}$ have positive finite values,
where $\mathbb{D}$ is open subset of
$\mathbb{R}^m$. Then the RCC of $f_i$
over $\mathbb{D}$ satisfies
$$
\begin{array}{rl}
&\min_{\{\alpha_i:
\alpha_i>0,\sum_i\alpha_i=1\}}\sum_i\frac{1}{\alpha_i}f_i(t)\nonumber\\
=&\sum_if_i(t)+\max_{g_{i,j}(t)}\sum_{i\not= j}g_{i,j}(t)
\end{array}
$$
subject to
$$
\begin{array}{rl}
g_{ij}:\mathbb{R}^m\rightarrow \mathbb{R},\ g_{j,i}(t)=g_{i,j}(t),\ \left[\begin{array}{cc}
f_i(t)&g_{i,j}(t)\\
g_{i,j}(t)&f_j(t)
\end{array}\right]\ge 0.
\end{array}
$$\end{lemma}

\section{Observer design}\label{2014-3-15-1512-section-1}

In this section, we will design a state observer (\ref{equation-7}) for GRN (\ref{equation-4}), that is, find a pair of observer gain matrices $K_{1}$ and $K_{2}$ such that
the trivial solution of system (\ref{equation-8}) is asymptotically stable under Dirichlet boundary conditions.
For this end, we define
$$e_{0}=0_{14n\times n},$$
$$e_{i}=\mathrm{col}(0_{n\times(i-1)n},I_{n},0_{n\times(n-i)n})^{T}, i\in\langle 14\rangle,$$
$$
\varphi(t,s,x)=\mathrm{col}\left(m(s,x),\int^{t}_{t-\bar{\tau}}m(s,x)\mathrm{d}s\right),
$$
$$
\psi(t,s,x)=\mathrm{col}\left(p(s,x),\int^{t}_{t-\bar{\sigma}}p(s,x)\mathrm{d}s\right),
$$
$$
\begin{array}{rl}
\varsigma(t,x)
=&\hspace{-3mm}\mathrm{col}(m(t,x),m(t-\overline{\tau},x),m(t-\tau(t),x),p(t,x)\\
&\hspace{-3mm}p(t-\overline{\sigma},x), p(t-\sigma(t),x), \bar{f}(p(t,x)), \\
&\hspace{-3mm}\bar{f}(p(t-\sigma(t),x)),\frac{\partial m(t,x)}{\partial t}, \frac{\partial p(t,x)}{\partial t},\\
&\hspace{-3mm}\frac{1}{\tau(t)}\int^{t}_{t-\tau(t)}m(s,x)\mathrm{d}s,\frac{1}{\overline{\tau}-\tau(t)}\int^{t-\tau(t)}_{t-\overline{\tau}}m(s,x)\mathrm{d}s,\\
&\hspace{-3mm}\frac{1}{\sigma(t)}\int^{t}_{t-\sigma(t)}p(s,x)\mathrm{d}s,\frac{1}{\overline{\sigma}-\sigma(t)}\int^{t-\sigma(t)}_{t-\overline{\sigma}}p(s,x)\mathrm{d}s).
\end{array}$$

\begin{theorem}\label{theorem1-1}
For given scalars $\overline{\tau}$, $\overline{\sigma}$, $\mu_{1}$ and $\mu_{2}$ satisfying (\ref{equation-2}), the trivial solution of error system (\ref{equation-8}) under Dirichlet boundary conditions is asymptotically stable if there exist matrices $Q_{i}^{T}=Q_{i}>0$ $(i\in\langle 5\rangle)$, $R_{j}^{T}=R_{j}>0$ $(j\in\langle 4\rangle)$, $M_{j}^{T}=M_{j}>0$ $(j\in \langle 2\rangle)$, diagonal matrices $P_{j}>0$, $\Lambda_{j}>0$ $(j\in \langle 2\rangle)$, and matrices $G_{1}$, $G_{2}$, $W_{1}$ and $W_{2}$ of appropriate sizes, such that the following LMIs hold for $\tau\in\{0, \bar{\tau}\}$ and $\sigma\in\{0, \bar{\sigma}\}$:
\begin{equation}\label{equation-10}
\hat{R}_j:=\left [\begin{array}{cc}
\tilde{R}_{j}&G_{j}\\
G^{T}_{j}&\tilde{R}_{j}
\end{array}\right]\geq0, \ \ \ \ \ \ j\in \langle 2\rangle,
\end{equation}
\begin{equation}\label{equation-11}
\Phi(\tau,\sigma)=\Phi_{0}+\Phi_{1}+\Phi_{2}(\tau,\sigma)+\Phi_{3}+\Phi_{4}(\tau,\sigma)+\Phi_{5}(\tau,\sigma)<0,
\end{equation}
where
$$\renewcommand{\arraystretch}{1.5}
\begin{array}{rl}\Phi_{0}=
&-2e_{7}\Lambda_{1}e_{7}^{T}+e_{4}\Lambda_{1}Ke_{7}^{T}+e_{7}K\Lambda_{1}e_{4}^{T}
-2e_{8}\Lambda_{2}e_{8}^{T}\\
&+e_{6}K\Lambda_{2}e_{8}^{T}+e_{8}\Lambda_{2}K e_{6}^{T}-e_{9}(P_{1}A+W_{1}M)e_{1}^{T}\\
&-e_{1}(P_{1}A+W_{1}M)^{T}e_{9}^{T}+e_{9}P_{1}We_{8}^{T}\\
&+e_{8}W^{T}P_{1}e_{9}^{T}-2e_{9}P_{1}e_{9}^{T}-e_{10}(P_{2}C+W_{2}N)e_{4}^{T}\\
&-e_{4}(P_{2}C+W_{2}N)^{T}e_{10}^{T}+e_{10}P_{2}Be_{3}^{T}\\
&+e_{3}B^{T}P_{2}e_{10}^{T}-2e_{10}P_{2}e_{10}^{T},
\end{array}
$$
$$\renewcommand{\arraystretch}{1.5}
\begin{array}{rl}\Phi_{1}=&-0.5\pi^{2}e_{1}P_{1}D_{L}e_{1}^{T}-2e_{1}(P_{1}A+W_{1}M)e_{1}^{T}\\
&+e_{1}P_{1}We_{8}^{T}+e_{8}W^{T}P_{1}e_{1}^{T}\\
&-0.5\pi^2e_{4}P_{2}D^{*}_{L}e_{4}^{T}-2e_{4}(P_{2}C+W_{2}N)e_{4}^{T}\\
&+e_{4}P_{2}Be_{3}^{T}+e_{3}B^{T}P_{2}e_{4}^{T},
\end{array}
$$
$$\renewcommand{\arraystretch}{1.5}
\begin{array}{rl}
\Phi_{2}(\tau,\sigma)=&e_{1}Q_{1}e_{1}^{T}-(1-\mu_{1})e_{3}Q_{1}e_{3}^{T}\\
&+e_{4}Q_{3}e_{4}^{T}-(1-\mu_{2})e_{6}Q_{3}e_{6}^{T}\\
&+\Delta_{1}Q_{2}\Delta_{1}^{T}+\tau(\Delta_{1}Q_{2}\Delta_{2}^{T}+\Delta_{2}Q_{2}\Delta_{1}^{T})\\&
-\Delta_{3}Q_{2}\Delta_{3}^{T}-\tau(\Delta_{3}Q_{2}\Delta_{2}^{T}+\Delta_{2}Q_{2}\Delta_{3}^{T})\\
&+\Delta_{4}Q_{2}\Delta_{6}^{T}+\Delta_{6}Q_{2}\Delta_{4}^{T}\\
&+\tau(\Delta_{5}Q_{2}\Delta_{6}^{T}+\Delta_{6}Q_{2}\Delta_{5}^{T})
+\Theta_{1}Q_{4}\Theta_{1}^{T}\\
&+\sigma(\Theta_{1}Q_{4}\Theta_{2}^{T}+\Theta_{2}Q_{4}\Theta_{1}^{T})-\Theta_{3}Q_{4}\Theta_{3}^{T}\\&
-\sigma(\Theta_{3}Q_{4}\Theta_{2}^{T}+\Theta_{2}Q_{4}\Theta_{3}^{T})+\Theta_{4}Q_{4}\Theta_{6}^{T}\\&+\Theta_{6}Q_{4}\Theta_{4}^{T}
+\sigma(\Theta_{5}Q_{4}\Theta_{6}^{T}+\Theta_{6}Q_{4}\Theta_{5}^{T}),
\end{array}
$$
$$
\Phi_{3}=e_{7}Q_{5}e_{7}^{T}-(1-\mu_{2})e_{8}Q_{5}e_{8}^{T},
$$
$$\renewcommand{\arraystretch}{1.5}
\begin{array}{rl}
\Phi_{4}(\tau,\sigma)\hspace*{-1mm}= &\hspace*{-3mm}\Phi_{41}-\Phi_{42}(\tau)-\Phi_{43}(\sigma)\\
&\hspace*{-3mm}-[\Delta_{7} \ \Delta_{8}]\hat{R}_1[\Delta_{7} \ \Delta_{8}]^{T}-[\Theta_{7} \ \Theta_{8}]\hat{R}_2[\Theta_{7} \ \Theta_{8}]^{T},
\end{array}
$$
$$
\Phi_{41}=\bar{\tau}^{2}e_{9}R_{1}e_{9}^{T}+\bar{\sigma}^{2}e_{10}R_{2}e_{10}^{T}+\bar{\tau}^{2}e_{1}R_{3}e_{1}^{T}+\bar{\sigma}^{2}e_{4}R_{4}e_{4}^{T},
$$
$$
\Phi_{42}(\tau)=\bar{\tau}(\bar{\tau}-\tau)e_{12}R_{3}e_{12}^{T}+\bar{\tau}\tau e_{11}R_{3}e_{11}^{T},
$$
$$
\Phi_{43}(\tau)=\bar{\sigma}(\bar{\sigma}-\sigma)e_{14}R_{4}e_{14}^{T}+\bar{\sigma}\sigma e_{13}R_{4}e_{13}^{T},
$$
$$\renewcommand{\arraystretch}{1.5}
\begin{array}{rl}
\Phi_{5}(\tau,\sigma)=&\Phi_{51}-\Phi_{52}-\Phi_{53}\\
&-\frac{(\overline{\tau}-\tau)}{\overline{\tau}}\Delta_{8}\tilde{M}_{1}\Delta^{T}_{8}-\frac{(\overline{\sigma}-\sigma)}{\overline{\sigma}}\Theta_{8}\tilde{M}_{2}\Theta^{T}_{8},
\end{array}
$$
$$\renewcommand{\arraystretch}{1.5}
\Phi_{51}=\frac{\bar{\tau}^{2}}{2}e_{9}M_{1}e_{9}^{T}+\frac{\bar{\sigma}^{2}}{2}e_{10}M_{2}e_{10}^{T},
$$
$$\renewcommand{\arraystretch}{1.5}
\begin{array}{rl}
\Phi_{52}=&(e_{1}-e_{11})M_{1}(e_{1}-e_{11})^{T}\\
&+(e_{3}-e_{12})M_{1}(e_{3}-e_{12})^{T},
\end{array}
$$
$$\renewcommand{\arraystretch}{1.5}
\begin{array}{rl}
\Phi_{53}=&(e_{4}-e_{13})M_{2}(e_{4}-e_{13})^{T}\\
&+(e_{6}-e_{14})M_{2}(e_{6}-e_{14})^{T},
\end{array}
$$
$$
\Delta_{1}=[e_{1} \ \ \bar{\tau} e_{12}],\
\Delta_{2}=[e_{0} \ \ e_{11}-e_{12}],
$$
$$
\Delta_{3}=[e_{2} \ \ \bar{\tau} e_{12}],\
\Delta_{4}=[\bar{\tau} e_{12} \ \ \bar{\tau}^{2}e_{12}],
$$
$$
\Delta_{5}=[e_{11}-e_{12} \ \ \bar{\tau}(e_{11}-e_{12})],\
\Delta_{6}=[e_{0} \ \ e_{1}-e_{2}],
$$
$$
\Delta_{7}=[e_{3}-e_{2} \ \ e_{3}+e_{2}-2e_{12} ],
$$
$$
\Delta_{8}=[e_{1}-e_{3} \ \ e_{1}+e_{3}-2e_{11}],
$$
$$
\Theta_{1}=[e_{4} \ \ \bar{\sigma} e_{14}],\
\Theta_{2}=[e_{0} \ \ e_{13}-e_{14}],
$$
$$
\Theta_{3}=[e_{5} \ \ \bar{\sigma} e_{14}],\
\Theta_{4}=[\bar{\sigma} e_{14} \ \ \bar{\sigma}^{2}e_{14}],
$$
$$
\Theta_{5}=[e_{13}-e_{14} \ \ \bar{\sigma}(e_{13}-e_{14})],\
\Theta_{6}=[e_{0} \ \ e_{4}-e_{5}],
$$
$$
\Theta_{7}=[e_{6}-e_{5} \ \ e_{6}+e_{5}-2e_{14} ],
$$
$$
\Theta_{8}=[e_{4}-e_{6} \ \ e_{4}+e_{6}-2e_{13}],
$$
$$
\tilde{R}_{1}=\mathrm{diag}(R_{1},3R_{1}),\
\tilde{R}_{2}=\mathrm{diag}(R_{2},3R_{2}),
$$
$$
\tilde{M}_{1}=\frac{1}{\bar{\tau}}\mathrm{diag}(M_{1},3M_{1}),\
\tilde{M}_{2}=\frac{1}{\bar{\sigma}}\mathrm{diag}(M_{2},3M_{2}),
$$
$$D_{L}=\mathrm{diag}\left(\sum^{l}_{k=1}\frac{D_{1k}}{L^{2}_{k}},\sum^{l}_{k=1}\frac{D_{2k}}{L^{2}_{k}},\ldots,\sum^{l}_{k=1}\frac{D_{nk}}{L^{2}_{k}}\right),$$
$$D^{*}_{L}=\mathrm{diag}\left(\sum^{l}_{k=1}\frac{D^{*}_{1k}}{L^{2}_{k}},\sum^{l}_{k=1}\frac{D^{*}_{2k}}{L^{2}_{k}},\ldots,\sum^{l}_{k=1}\frac{D^{*}_{nk}}{L^{2}_{k}}\right),$$
and $L_{k}$, $D_{ik}$, $D^{*}_{ik}$, $A$, $B$, $C$, $W$ and $K$ are the same with previous ones.

Moreover, the observer gain matrices are given by $K_{1}=P_{1}^{-1}W_{1}$ and $K_{2}=P_{2}^{-1}W_{2}$.
\end{theorem}

\begin{proof}
Construct a Lyapunov-Krasovskii functional for error system (\ref{equation-8}) as follows:
$$V(t,m,p)=\sum^{5}_{i=1}V_{i}(t,m,p),$$where
$$
\begin{array}{rl}
V_{1}(t,m,p)=
&\hspace{-3mm}\int_{\Omega}m^{T}(t,x)P_{1}m(t,x)\mathrm{d}x\\
&\hspace{-3mm}+\int_{\Omega}p^{T}(t,x)P_{2}p(t,x)\mathrm{d}x\\
&\hspace{-3mm}+\sum_{k=1}^l\int_{\Omega}\frac{\partial m^{T}(t,x)}{\partial x_k}P_1D_{k}\frac{\partial m(t,x)}{\partial x_k}\mathrm{d}x\\
&\hspace{-3mm}+\sum_{k=1}^l\int_{\Omega}\frac{\partial p^{T}(t,x)}{\partial x_k}P_2D^{*}_{k}\frac{\partial p(t,x)}{\partial x_k}\mathrm{d}x,
\end{array}
$$
$$
\begin{array}{rl}
V_{2}(t,m,p)=&\hspace{-3mm}\int_{\Omega}\int^{t}_{t-\tau(t)}m^{T}(s,x)Q_{1}m(s,x)\mathrm{d}s\mathrm{d}x\\
&\hspace{-3mm}+\int_{\Omega}\int^{t}_{t-\overline{\tau}}\varphi^{T}(t,s,x)Q_{2}\varphi(t,s,x)\mathrm{d}s\mathrm{d}x\\
&\hspace{-3mm}+\int_{\Omega}\int^{t}_{t-\sigma(t)}p^{T}(s,x)Q_{3}p(s,x)\mathrm{d}s\mathrm{d}x\\
&\hspace{-3mm}+\int_{\Omega}\int^{t}_{t-\overline{\sigma}}\psi^{T}(t,s,x)Q_{4}\psi(t,s,x)\mathrm{d}s\mathrm{d}x,
\end{array}
$$
$$V_{3}(t,m,p)=\int_{\Omega}\int^{t}_{t-\sigma(t)}\bar{f}^{T}(p(s,x))Q_{5}\bar{f}(p(s,x))\mathrm{d}s\mathrm{d}x,$$
$$
\begin{array}{rl}
V_{4}(t,m,p)=&\hspace{-3mm}\overline{\tau}\int_{\Omega}\int^{0}_{-\overline{\tau}}\int^{t}_{t+\theta}\frac{\partial m^{T}(s,x)}{\partial s}R_{1}\frac{\partial m(s,x)}{\partial s}\mathrm{d}s\mathrm{d}\theta\mathrm{d}x
\\
&\hspace{-3mm}+\overline{\sigma}\int_{\Omega}\int^{0}_{-\overline{\sigma}}\int^{t}_{t+\theta}\frac{\partial p^{T}(s,x)}{\partial s}R_{2}\frac{\partial p(s,x)}{\partial s}\mathrm{d}s\mathrm{d}\theta\mathrm{d}x\\
&\hspace{-3mm}+\overline{\tau}\int_{\Omega}\int^{0}_{-\overline{\tau}}\int^{t}_{t+\theta}m^{T}(s,x)R_{3}m(s,x)\mathrm{d}s\mathrm{d}\theta\mathrm{d}x
\\&\hspace{-3mm}+\overline{\sigma}\int_{\Omega}\int^{0}_{-\overline{\sigma}}\int^{t}_{t+\theta}p^{T}(s,x)R_{4}p(s,x)\mathrm{d}s\mathrm{d}\theta\mathrm{d}x,
\end{array}
$$
$$
\begin{array}{rl}
V_{5}(t,m,p)=&\hspace{-3mm}\int_{\Omega}\int^{0}_{-\overline{\tau}}\int^{0}_{\theta}\int^{t}_{t+\lambda}\hspace{-2mm}\frac{\partial m^{T}(s,x)}{\partial s}M_{1}\frac{\partial m(s,x)}{\partial s}\mathrm{d}s\mathrm{d}\lambda\mathrm{d}\theta\mathrm{d}x\\
&\hspace{-3mm}\int_{\Omega}\int^{0}_{-\overline{\sigma}}\int^{0}_{\theta}\int^{t}_{t+\lambda}\hspace{-2mm}\frac{\partial p^{T}(s,x)}{\partial s}M_{2}\frac{\partial p(s,x)}{\partial s}\mathrm{d}s\mathrm{d}\lambda\mathrm{d}\theta\mathrm{d}x.
\end{array}
$$
Taking the time derivatives of $V_{i}(t,m, p)$ $(i\in\langle5\rangle)$ along the trajectory of error system (\ref{equation-8}) yields
\begin{equation}\label{equation-12}
\begin{array}{rl}
&\frac{\partial}{\partial t}V_{1}(t,m,p)\\
=&2\int_{\Omega}m^{T}(t,x)P_{1}\Big[\sum^{l}_{k=1}\frac{\partial}{\partial x_{k}}\left(D_{k}\frac{\partial m(t,x)}{\partial x_{k}}\right)\\&-(A+K_{1}M)m(t,x)+W\bar{f}(p(t-\sigma(t),x))\Big]\mathrm{d}x\\&+2\int_{\Omega}p^{T}(t,x)P_{2}\Big[\sum^{l}_{k=1}\frac{\partial}{\partial x_{k}}\left(D^{*}_{k}\frac{\partial p(t,x)}{\partial x_{k}}\right)\\&-(C+K_{2}N)p(t,x)+Bm(t-\tau(t),x)\Big]\mathrm{d}x
\\&+2\sum^{l}_{k=1}\int_{\Omega}\frac{\partial m^{T}(t,x)}{\partial x_{k}}P_{1}D_{k}\frac{\partial }{\partial x_{k}}(\frac{\partial m(t,x)}{\partial t})\mathrm{d}x
\\&
+2\sum^{l}_{k=1}\int_{\Omega}\frac{\partial p^{T}(t,x)}{\partial x_{k}}P_{2}D^{*}_{k}\frac{\partial }{\partial x_{k}}(\frac{\partial p(t,x)}{\partial t})\mathrm{d}x,
\end{array}
\end{equation}
\begin{equation}\label{equation-13}\renewcommand{\arraystretch}{1.5}
\begin{array}{rl}
&\frac{\partial}{\partial t}V_{2}(t,m,p)\\
=&\int_{\Omega}m^{T}(t,x)Q_{1}m(t,x)\mathrm{d}x\\&-(1-\dot{\tau}(t))\int_{\Omega}m^{T}(t-\tau(t),x) Q_{1}m(t-\tau(t),x)\mathrm{d}x\\&+\int_{\Omega}p^{T}(t,x)Q_{3}p(t,x)\mathrm{d}x\\
&-(1-\dot{\sigma}(t))\int_{\Omega}p^{T}(t-\sigma(t),x)Q_{3}p(t-\sigma(t),x)\mathrm{d}x\\&
+\int_{\Omega}\varphi^{T}(t,t,x)Q_{2}\varphi(t,t,x)\mathrm{d}x
\\&-\int_{\Omega}\varphi^{T}(t,t-\overline{\tau},x)Q_{2}\varphi(t,t-\overline{\tau},x)\mathrm{d}x\\&
+2\int_{\Omega}\int^{t}_{t-\overline{\tau}}\varphi^{T}(t,s,x)Q_{2}\frac{\partial \varphi(t,s,x)}{\partial t}\mathrm{d}s\mathrm{d}x
\\&
+\int_{\Omega}\psi^{T}(t,t,x)Q_{4}\psi(t,t,x)\mathrm{d}x
\\&-\int_{\Omega}\psi^{T}(t,t-\overline{\sigma},x)Q_{4}\psi(t,t-\overline{\sigma},x)\mathrm{d}x,
\\&
+2\int_{\Omega}\int^{t}_{t-\overline{\sigma}}\psi^{T}(t,s,x)Q_{4}\frac{\partial \psi(t,s,x)}{\partial t}\mathrm{d}s\mathrm{d}x\\
\leq&\int_{\Omega}\varsigma^{T}(t,x)\Phi_{2}(\tau(t),\sigma(t))
\varsigma(t,x)\mathrm{d}x,
\end{array}
\end{equation}
\begin{equation}\label{equation-14}\renewcommand{\arraystretch}{1.5}
\begin{array}{rl}
&\frac{\partial}{\partial t}V_{3}(t,m,p)\\
=&\hspace{-3mm}-(1-\dot{\sigma}(t))\hspace{-2mm}\int_{\Omega}\bar{f}^{T}(p(t-\sigma(t),x)) Q_{5}\bar{f}(p(t-\sigma(t),x))\mathrm{d}x\\
&\hspace{-3mm}+\int_{\Omega}\bar{f}^{T}(p(t,x))Q_{5}\bar{f}(p(t,x))\mathrm{d}x
\\
\leq&\hspace{-3mm}\int_{\Omega}\varsigma^{T}(t,x)
\Phi_{3}
\varsigma(t,x)\mathrm{d}x,
\end{array}
\end{equation}
\begin{equation}\label{equation-15}\renewcommand{\arraystretch}{1.5}
\begin{array}{rl}
\frac{\partial}{\partial t}V_{4}(t,m,p)=&\overline{\tau}^{2}\int_{\Omega} \frac{\partial m^{T}(t,x)}{\partial t}R_{1}\frac{\partial m(t,x)}{\partial t}\mathrm{d}x
\\&-\overline{\tau}\int_{\Omega}\int^{t}_{t-\overline{\tau}}\frac{\partial m^{T}(s,x)}{\partial s}R_{1}\frac{\partial m(s,x)}{\partial s}\mathrm{d}s\mathrm{d}x
\\&+\overline{\sigma}^{2}\int_{\Omega} \frac{\partial p^{T}(t,x)}{\partial t}R_{2}\frac{\partial p(t,x)}{\partial t}\mathrm{d}x
\\&-\overline{\sigma}\int_{\Omega}\int^{t}_{t-\overline{\sigma}}\frac{\partial p^{T}(s,x)}{\partial s}R_{2}\frac{\partial p(s,x)}{\partial s}\mathrm{d}s\mathrm{d}x\\
&+\overline{\tau}^{2}\int_{\Omega}m^{T}(t,x)R_{3} m(t,x)\mathrm{d}x
\\&-\overline{\tau}\int_{\Omega}\int^{t}_{t-\overline{\tau}}m^{T}(s,x)R_{3} m(s,x)\mathrm{d}s\mathrm{d}x
\\&+\overline{\sigma}^{2}\int_{\Omega} p^{T}(t,x)R_{4} p(t,x)\mathrm{d}x
\\&-\overline{\sigma}\int_{\Omega}\int^{t}_{t-\overline{\sigma}}p^{T}(s,x)R_{4} p(s,x)\mathrm{d}s\mathrm{d}x,
\end{array}
\end{equation}
\begin{equation}\label{equation-16}\renewcommand{\arraystretch}{1.5}
\begin{array}{rl}
\frac{\partial}{\partial t}V_{5}(t,m,p)=
&\hspace*{-3mm}\frac{\overline{\tau}^{2}}{2}\int_{\Omega} \frac{\partial m^{T}(t,x)}{\partial t}M_{1}\frac{\partial m(t,x)}{\partial t}\mathrm{d}x\\
&\hspace*{-3mm}-\int_{\Omega}\int^{0}_{-\overline{\tau}}\int^{t}_{t+\theta}\frac{\partial m^{T}(s,x)}{\partial s}M_{1}\frac{\partial m(s,x)}{\partial s}\mathrm{d}s\mathrm{d}\theta\mathrm{d}x\\
&\hspace*{-3mm}+\frac{\overline{\sigma}^{2}}{2}\int_{\Omega} \frac{\partial p^{T}(t,x)}{\partial t}M_{2}\frac{\partial p(t,x)}{\partial t}\mathrm{d}x\\
&\hspace*{-3mm}-\int_{\Omega}\int^{0}_{-\overline{\sigma}}\int^{t}_{t+\theta}\frac{\partial p^{T}(s,x)}{\partial s}M_{2}\frac{\partial p(s,x)}{\partial s}\mathrm{d}s\mathrm{d}\theta\mathrm{d}x.
\end{array}
\end{equation}

From Green formula, Dirichlet boundary conditions and Lemma \ref{lemma2}, we have
\begin{equation}\label{equation-17}\renewcommand{\arraystretch}{1.5}
\begin{array}{rl}
&2\sum^{l}_{k=1}\int_{\Omega}m^{T}(t,x)P_{1}\frac{\partial}{\partial x_{k}}\left(D_{k}\frac{\partial m(t,x)}{\partial x_{k}}\right)\mathrm{d}x\\=&2\sum^{l}_{k=1}\int_{\Omega}\frac{\partial}{\partial x_{k}}\left(m^{T}(t,x)P_{1}D_{k}\frac{\partial m(t,x)}{\partial x_{k}}\right)\mathrm{d}x\\&-2\sum^{l}_{k=1}\int_{\Omega}\frac{\partial m^{T}(t,x)}{\partial x_{k}}P_{1}D_{k}\frac{\partial m(t,x)}{\partial x_{k}}\mathrm{d}x\\=&2\sum^{l}_{k=1}\int_{\partial\Omega}\left(m^{T}(t,x)P_{1}D_{k}\frac{\partial m(t,x)}{\partial x_{k}}\right)^{l}_{k=1}\cdot \overline{n} \ \mathrm{d}s\\&-2\sum^{l}_{k=1}\int_{\Omega}\frac{\partial m^{T}(t,x)}{\partial x_{k}}P_{1}D_{k}\frac{\partial m(t,x)}{\partial x_{k}}\mathrm{d}x\\=&-2\sum^{l}_{k=1}\int_{\Omega}\frac{\partial m^{T}(t,x)}{\partial x_{k}}P_{1}D_{k}\frac{\partial m(t,x)}{\partial x_{k}}\mathrm{d}x\\\leq
&-\frac{\pi^{2}}{2}\int_{\Omega}m^{T}(t,x)P_{1}D_{L}m(t,x)\mathrm{d}x,
\end{array}
\end{equation}
where
$$ \renewcommand{\arraystretch}{2}
\begin{array}{rl}
&\left(m^{T}(t,x)P_{1}D_{k}\frac{\partial m(t,x)}{\partial x_{k}}\right)^{l}_{k=1}\\=&
\left(m^{T}(t,x)P_{1}D_{1}\frac{\partial m(t,x)}{\partial x_{1}}, \ldots, m^{T}(t,x)P_{1}D_{l}\frac{\partial m(t,x)}{\partial x_{l}}\right).
\end{array}
$$
Similarly, \begin{equation}\label{equation-18}
\begin{array}{rl}
&2\sum^{l}_{k=1}\int_{\Omega}p^{T}(t,x)P_{2}\frac{\partial}{\partial x_{k}}\left(D^{*}_{k}\frac{\partial p(t,x)}{\partial x_{k}}\right)\mathrm{d}x
\\\leq&-\frac{\pi^{2}}{2}\int_{\Omega}p^{T}(t,x)P_{2}D^{*}_{L}p(t,x)\mathrm{d}x.
\end{array}
\end{equation}
The combination of (\ref{equation-12}), (\ref{equation-17}) and (\ref{equation-18}) gives
\begin{equation}\label{equation-19}\renewcommand{\arraystretch}{1.5}
\begin{array}{rl}
&\frac{\partial}{\partial t}V_{1}(t,m,p)\\
=&2\int_{\Omega}m^{T}(t,x)P_{1}\left[-\frac{\pi^{2}}{4}D_{L}m(t,x)\right.\\&\left.-(A+K_{1}M)m(t,x)+W\bar{f}(p(t-\sigma(t),x))\right]\mathrm{d}x\\&+2\int_{\Omega}p^{T}(t,x)P_{2}\left[-\frac{\pi^{2}}{4}D^{*}_{L}p(t,x)\right.\\&\left.-(C+K_{2}N)p(t,x)+Bm(t-\tau(t),x)\right]\mathrm{d}x
\\&+2\sum^{l}_{k=1}\int_{\Omega}\frac{\partial m^{T}(t,x)}{\partial x_{k}}P_{1}D_{k}\frac{\partial }{\partial x_{k}}(\frac{\partial m(t,x)}{\partial t})\mathrm{d}x
\\&
+2\sum^{l}_{k=1}\int_{\Omega}\frac{\partial p^{T}(t,x)}{\partial x_{k}}P_{2}D^{*}_{k}\frac{\partial }{\partial x_{k}}(\frac{\partial p(t,x)}{\partial t})\mathrm{d}x\\
=&\int_{\Omega}\varsigma^{T}(t,x)\Phi_{1}
\varsigma(t,x)\mathrm{d}x\\&+2\sum^{l}_{k=1}\int_{\Omega}\frac{\partial m^{T}(t,x)}{\partial x_{k}}P_{1}D_{k}\frac{\partial }{\partial x_{k}}(\frac{\partial m(t,x)}{\partial t})\mathrm{d}x
\\&
+2\sum^{l}_{k=1}\int_{\Omega}\frac{\partial p^{T}(t,x)}{\partial x_{k}}P_{2}D^{*}_{k}\frac{\partial }{\partial x_{k}}(\frac{\partial p(t,x)}{\partial t})\mathrm{d}x.                                                                                                                                                        \end{array}
\end{equation}

Note that the second term on the right of (\ref{equation-15}) can be written as:
\begin{equation}\label{equation-20}
\begin{array}{rl}
&-\overline{\tau}\int_{\Omega}\int^{t}_{t-\overline{\tau}}\frac{\partial m^{T}(s,x)}{\partial s}R_{1}\frac{\partial m(s,x)}{\partial s}\mathrm{d}s\mathrm{d}x\\
=&-\overline{\tau}\int_{\Omega}\int^{t-\tau(t)}_{t-\overline{\tau}}\frac{\partial m^{T}(s,x)}{\partial s}R_{1}\frac{\partial m(s,x)}{\partial s}\mathrm{d}s\mathrm{d}x\\
&-\overline{\tau}\int_{\Omega}\int^{t}_{t-\tau(t)}\frac{\partial m^{T}(s,x)}{\partial s}R_{1}\frac{\partial m(s,x)}{\partial s}\mathrm{d}s\mathrm{d}x.
\end{array}
\end{equation}
Applying Lemma \ref{2015-lemma-1}, one can obtain that
$$
\begin{array}{rl}
&-\overline{\tau}\int_{\Omega}\int^{t-\tau(t)}_{t-\overline{\tau}}\frac{\partial m^{T}(s,x)}{\partial s}R_{1}\frac{\partial m(s,x)}{\partial s}\mathrm{d}s\mathrm{d}x\\
\leq&-\frac{\overline{\tau}}{\overline{\tau}-\tau(t)}\int_{\Omega}\varsigma^{T}(t,x)\Delta_{7}\tilde{R}_{1}\Delta_{7}^{T}\varsigma(t,x)\mathrm{d}x
\end{array}
$$
and
\begin{equation}\label{equation-22}\renewcommand{\arraystretch}{1.5}
\begin{array}{rl}
&-\overline{\tau}\int_{\Omega}\int^{t}_{t-\tau(t)}\frac{\partial m^{T}(s,x)}{\partial s}R_{1}\frac{\partial m(s,x)}{\partial s}\mathrm{d}s\mathrm{d}x\\
\leq&-\frac{\overline{\tau}}{\tau(t)}\int_{\Omega}\varsigma^{T}(t,x)\Delta_{8}\tilde{R}_{1}\Delta_{8}\varsigma(t,x)\mathrm{d}x.
\end{array}
\end{equation}
This, together with (\ref{equation-10}), (\ref{equation-20}) and Lemma \ref{lemma6}, implies that
\begin{equation}\label{equation-23}\renewcommand{\arraystretch}{1.5}
\begin{array}{rl}
&-\overline{\tau}\int_{\Omega}\int^{t}_{t-\overline{\tau}}\frac{\partial m^{T}(s,x)}{\partial s}R_{1}\frac{\partial m(s,x)}{\partial s}\mathrm{d}s\mathrm{d}x\\
\leq&-\int_{\Omega}\varsigma^{T}(t,x) [\Delta_{7} \ \Delta_{8}] \hat{R}_1[\Delta_{7} \ \Delta_{8}]^{T} \varsigma(t,x)\mathrm{d}x.
\end{array}
\end{equation}
Similarly,
\begin{equation}\label{equation-24}\renewcommand{\arraystretch}{1.5}
\begin{array}{rl}
&-\overline{\sigma}\int_{\Omega}\int^{t}_{t-\overline{\sigma}}\frac{\partial p^{T}(s,x)}{\partial s}R_{2}\frac{\partial p(s,x)}{\partial s}\mathrm{d}s\mathrm{d}x\\
\leq&-\int_{\Omega}\varsigma^{T}(t,x) [\Theta^{T}_{7} \ \Theta_{8}]\hat{R}_2[\Theta^{T}_{7} \ \Theta_{8}]^{T} \varsigma(t,x)\mathrm{d}x.
\end{array}
\end{equation}
On the other hand, by Lemma \ref{lemma3}, it yields
\begin{equation}\label{equation-25}\renewcommand{\arraystretch}{1.5}
\begin{array}{rl}
&\hspace*{-3mm}-\overline{\tau}\int_{\Omega}\int^{t}_{t-\overline{\tau}}m^{T}(s,x)R_{3} m(s,x)\mathrm{d}s\mathrm{d}x\\
=&\hspace*{-3mm}-\overline{\tau}\int_{\Omega}\int^{t-\tau(t)}_{t-\overline{\tau}}m^{T}(s,x)R_{3}m(s,x)\mathrm{d}s\mathrm{d}x\\
&\hspace*{-3mm}-\overline{\tau}\int_{\Omega}\int^{t}_{t-\tau(t)}m^{T}(s,x)R_{3}m(s,x)\mathrm{d}s\mathrm{d}x\\
\leq&\hspace*{-3mm}-\frac{\overline{\tau}}{\overline{\tau}-\tau(t)}\int_{\Omega}\int^{t-\tau(t)}_{t-\overline{\tau}} m^{T}(s,x)\mathrm{d}sR_{3}\int^{t-\tau(t)}_{t-\overline{\tau}} m(s,x)\mathrm{d}s\mathrm{d}x\\
&\hspace*{-3mm}-\frac{\overline{\tau}}{\tau(t)}\int_{\Omega}\int^{t}_{t-\tau(t)} m^{T}(s,x)\mathrm{d}sR_{3}\int^{t}_{t-\tau(t)} m^{T}(s,x)\mathrm{d}s\mathrm{d}x\\
\leq&\hspace*{-3mm}-\int_{\Omega}\varsigma^{T}(t,x)
\Phi_{42}(\tau(t))
\varsigma(t,x)\mathrm{d}x.
\end{array}
\end{equation}
In the same way,
\begin{equation}\label{equation-26}\renewcommand{\arraystretch}{1.5}
\begin{array}{rl}
&-\overline{\sigma}\int_{\Omega}\int^{t}_{t-\overline{\sigma}}p^{T}(s,x)R_{4} p(s,x)\mathrm{d}s\mathrm{d}x\\
\leq&-\int_{\Omega}\varsigma^{T}(t,x)\Phi_{43}(\sigma(t))
\varsigma(t,x)\mathrm{d}x.
\end{array}
\end{equation}
Combining (\ref{equation-15}), (\ref{equation-23})-(\ref{equation-26}) yields
\begin{equation}\label{equation-27}\renewcommand{\arraystretch}{1.5}
\begin{array}{rl}
\frac{\partial}{\partial t}V_{4}(t,m,p)\leq&\int_{\Omega}\varsigma^{T}(t,x) \Phi_{4}(\tau(t),\sigma(t))
\varsigma(t,x)\mathrm{d}x.
\end{array}
\end{equation}

The second term on the right of (\ref{equation-16}) can be divided into three parts:
\begin{equation}\label{equation-28}\renewcommand{\arraystretch}{1.5}
\begin{array}{rl}
 &-\int_{\Omega}\int^{0}_{-\overline{\tau}}\int^{t}_{t+\theta}\frac{\partial m^{T}(s,x)}{\partial s}M_{1}\frac{\partial m(s,x)}{\partial s}\mathrm{d}s\mathrm{d}\theta\mathrm{d}x\\
=&-\int_{\Omega}\int^{0}_{-\tau(t)}\int^{t}_{t+\theta}\frac{\partial m^{T}(s,x)}{\partial s}M_{1}\frac{\partial m(s,x)}{\partial s}\mathrm{d}s\mathrm{d}\theta\mathrm{d}x\\&
-\int_{\Omega}\int^{-\tau(t)}_{-\overline{\tau}}\int^{t-\tau(t)}_{t+\theta}\frac{\partial m^{T}(s,x)}{\partial s}M_{1}\frac{\partial m(s,x)}{\partial s}\mathrm{d}s\mathrm{d}\theta\mathrm{d}x\\
&-(\overline{\tau}-\tau(t))\int_{\Omega}\int^{t}_{t-\tau(t)}\frac{\partial m^{T}(s,x)}{\partial s}M_{1}\frac{\partial m(s,x)}{\partial s}\mathrm{d}s\mathrm{d}x.
\end{array}
\end{equation}
By using Lemma \ref{lemma3}, we can estimate the following inequalities
\begin{equation}\label{equation-29}\renewcommand{\arraystretch}{1.5}
\begin{array}{rl}
 &-\int_{\Omega}\int^{0}_{-\tau(t)}\int^{t}_{t+\theta}\frac{\partial m^{T}(s,x)}{\partial s}M_{1}\frac{\partial m(s,x)}{\partial s}\mathrm{d}s\mathrm{d}\theta\mathrm{d}x
\\&-\int_{\Omega}\int^{-\tau(t)}_{-\overline{\tau}}\int^{t-\tau(t)}_{t+\theta}\frac{\partial m^{T}(s,x)}{\partial s}M_{1}\frac{\partial m(s,x)}{\partial s}\mathrm{d}s\mathrm{d}\theta\mathrm{d}x\\
\leq&-\frac{2}{\tau(t)^{2}}\int_{\Omega}\int^{0}_{-\tau(t)}\int^{t}_{t+\theta}\frac{\partial m^{T}(s,x)}{\partial s}\mathrm{d}s\mathrm{d}\theta\mathrm{d}x M_{1}\\
&\times\int^{0}_{-\tau(t)}\int^{t}_{t+\theta}\frac{\partial m(s,x)}{\partial s}\mathrm{d}s\mathrm{d}\theta\mathrm{d}x\\&
-\frac{2}{(\overline{\tau}-\tau(t))^{2}}\int_{\Omega}\int^{-\tau(t)}_{-\overline{\tau}}\int^{t-\tau(t)}_{t+\theta}\frac{\partial m^{T}(s,x)}{\partial s}\mathrm{d}s\mathrm{d}\theta\mathrm{d}xM_{1}\\&\times
\int_{\Omega}\int^{-\tau(t)}_{-\overline{\tau}}\int^{t-\tau(t)}_{t+\theta}\frac{\partial m(s,x)}{\partial s}\mathrm{d}s\mathrm{d}\theta\mathrm{d}x\\
=&
-2\int_{\Omega}\varsigma^{T}(t,x)\Phi_{52}\varsigma(t,x)\mathrm{d}x.
\end{array}
\end{equation}
As in (\ref{equation-22}), the last term on the right of (\ref{equation-28}) can be bounded by applying the same procedure,
\begin{equation}\label{equation-30}\renewcommand{\arraystretch}{1.5}
\begin{array}{rl}
&-(\overline{\tau}-\tau(t))\int_{\Omega}\int^{t}_{t-\tau(t)}\frac{\partial m^{T}(s,x)}{\partial s}M_{1}\frac{\partial m(s,x)}{\partial s}\mathrm{d}s\mathrm{d}x\\ \leq &
-\frac{(\overline{\tau}-\tau(t))}{\overline{\tau}}\int_{\Omega}\varsigma^{T}(t,x)\Delta_{8}\tilde{M}_{1}\Delta^{T}_{8}\varsigma(t,x)\mathrm{d}x.
\end{array}
\end{equation}
In a similar manner,
\begin{equation}\label{equation-31}\renewcommand{\arraystretch}{1.5}
\begin{array}{rl}
 &-\int_{\Omega}\int^{0}_{-\sigma(t)}\int^{t}_{t+\theta}\frac{\partial P^{T}(s,x)}{\partial s}M_{2}\frac{\partial P(s,x)}{\partial s}\mathrm{d}s\mathrm{d}\theta\mathrm{d}x
\\&-\int_{\Omega}\int^{-\sigma(t)}_{-\overline{\sigma}}\int^{t-\sigma(t)}_{t+\theta}\frac{\partial p^{T}(s,x)}{\partial s}M_{2}\frac{\partial p(s,x)}{\partial s}\mathrm{d}s\mathrm{d}\theta\mathrm{d}x\\\leq&
-2\int_{\Omega}\varsigma^{T}(t,x)\Phi_{53}\varsigma(t,x)\mathrm{d}x,
\end{array}
\end{equation}
and
\begin{equation}\label{equation-32}\renewcommand{\arraystretch}{1.5}
\begin{array}{rl}
&-(\overline{\sigma}-\sigma(t))\int_{\Omega}\int^{t}_{t-\sigma(t)}\frac{\partial p^{T}(s,x)}{\partial s}M_{2}\frac{\partial p(s,x)}{\partial s}\mathrm{d}s\mathrm{d}x\\ \leq &
-\frac{(\overline{\sigma}-\sigma(t))}{\overline{\sigma}}\int_{\Omega}\varsigma^{T}(t,x)\Theta_{8}\tilde{M}_{2}\Theta^{T}_{8}\varsigma(t,x)\mathrm{d}x.
\end{array}
\end{equation}
Combining (\ref{equation-16}) and (\ref{equation-28})-(\ref{equation-32}), we can obtain
\begin{equation}\label{equation-33}\renewcommand{\arraystretch}{1.5}
\frac{\partial}{\partial t}V_{5}(t,m,p)\leq\int_{\Omega}\varsigma^{T}(t,x)\Phi_{5}(\tau(t),\sigma(t))\varsigma(t,x)\mathrm{d}x.
\end{equation}

It is also easy to see that
 \begin{equation}\label{equation-36}
\begin{array}{rl}
&2\int_{\Omega}\frac{\partial m^{T}(t,x)}{\partial t}P_{1}\left[\sum^{l}_{k=1}\frac{\partial}{\partial x_{k}}\left(D_{k}\frac{\partial m(t,x)}{\partial x_{k}}\right)-\right.\\
&(A+K_{1}M)m(t,x)+W\bar{f}(p(t-\sigma(t),x)\\
&\left.-\frac{\partial m(t,x)}{\partial t}\right]\mathrm{d}x=0
\end{array}
\end{equation}
and
\begin{equation}\label{equation-37}
\begin{array}{rl}
&2\int_{\Omega}\frac{\partial p^{T}(t,x)}{\partial t}P_{2}\left[\sum^{l}_{k=1}\frac{\partial}{\partial x_{k}}\left(D^{*}_{k}\frac{\partial p(t,x)}{\partial x_{k}}\right)-\right.\\&(C+K_{2}N)p(t,x)+Bm(t-\tau(t),x)\\
&\left.-\frac{\partial p(t,x)}{\partial t}\right]\mathrm{d}x=0.
\end{array}
\end{equation}
According to Lemma \ref{lemma5}, Green formula and Dirichlet boundary conditions, we have
\begin{equation}\label{equation-38}\renewcommand{\arraystretch}{1.5}
\begin{array}{rl}
&\hspace*{-3mm}2\int_{\Omega}\frac{\partial m^{T}(t,x)}{\partial t}P_{1}\sum^{l}_{k=1}\frac{\partial}{\partial x_{k}}\left(D_{k}\frac{\partial m(t,x)}{\partial x_{k}}\right)\mathrm{d}x\\
=&\hspace*{-3mm}2\int_{\Omega}m^{T}(t,x)P_{1}\sum^{l}_{k=1}\hspace*{-1mm}\frac{\partial}{\partial x_{k}}\left[D_{k}\frac{\partial}{\partial x_{k}}\left(\frac{\partial m(t,x)}{\partial t}\right)\right]\hspace*{-1mm}\mathrm{d}x\\
=&\hspace*{-3mm}-2\sum^{l}_{k=1}\int_{\Omega}\frac{\partial m^{T}(t,x)}{\partial x_{k}}P_{1}D_{k}\frac{\partial }{\partial x_{k}}(\frac{\partial m(t,x)}{\partial t})\mathrm{d}x.
\end{array}
\end{equation}
Similarly,
\begin{equation}\label{equation-39}\renewcommand{\arraystretch}{1.5}
\begin{array}{rl}
&2\int_{\Omega}\frac{\partial p^{T}(t,x)}{\partial t}P_{2}\sum^{l}_{k=1}\frac{\partial}{\partial x_{k}}\left(D^{*}_{k}\frac{\partial p(t,x)}{\partial x_{k}}\right)\mathrm{d}x\\=&2\int_{\Omega}p^{T}(t,x)P_{2}\sum^{l}_{k=1}\frac{\partial}{\partial x_{k}}\left[D^{*}_{k}\frac{\partial}{\partial x_{k}}\left(\frac{\partial p(t,x)}{\partial t}\right)\right]\mathrm{d}x\\=&-2\sum^{l}_{k=1}\int_{\Omega}\frac{\partial p^{T}(t,x)}{\partial x_{k}}P_{2}D^{*}_{k}\frac{\partial }{\partial x_{k}}(\frac{\partial p(t,x)}{\partial t})\mathrm{d}x.
\end{array}
\end{equation}

Finally, for the diagonal matrices $\Lambda_{1}>0$ and $\Lambda_{2}>0$, it can be obtained  from (\ref{equation-5}) that
\begin{equation}\label{equation-34}
\begin{array}{rl}
&2\bar{f}^{T}(p(t,x))\Lambda_{1} \bar{f}(p(t,x))\\-&2p^{T}(t,x)K\Lambda_{1} \bar{f}(p(t,x))\leq0,
\end{array}
\end{equation}
\begin{equation}\label{equation-35}
\begin{array}{rl}
&2\bar{f}^{T}(p(t-\sigma(t),x))\Lambda_{2} \bar{f}(p(t-\sigma(t),x))\\-&2p^{T}(t-\sigma(t),x)K\Lambda_{2} \bar{f}(p(t-\sigma(t),x))\leq0.
\end{array}
\end{equation}

From (\ref{equation-13}), (\ref{equation-14}), (\ref{equation-19}), (\ref{equation-27}) and (\ref{equation-33})-(\ref{equation-35}), one can obtain
$$
\begin{array}{rl}
\frac{\partial}{\partial t}V(t,m,p)&=\sum^{5}_{i=1}\frac{\partial}{\partial t}V_{i}(t,m,p)\\&\leq\int_{\Omega}\varsigma^{T}(t,x)\Phi(\tau(t),\sigma(t))\varsigma(t,x)\mathrm{d}x.\\
\end{array}
$$
Since $\Phi(\tau(t),\sigma(t))$ depends affinely on $\tau(t)$ and $\sigma(t)$, respectively, it follows from (\ref{equation-11}) that $\frac{\partial}{\partial t}V(t,m,p)<0$ for all $\tau(t)$ and $\sigma(t)$ satisfying (\ref{equation-2}). Therefore, the trivial solution of error system (\ref{equation-8}) is asymptotically stable. This completes the proof.
\end{proof}

We end the section by the following remarks on Theorem \ref{theorem1-1}.

\begin{remark}\label{15sep07-remark-1}
Compared with \cite{Zhou.Xu.Shen(2011),NCA-2011-507,001,Han.Zhang.Wang-CSSP}, the advantages of this paper are as follows:
\begin{enumerate}
  \item We introduce new integral items like
$$
\overline{\tau}\int_{\Omega}\int^{0}_{-\overline{\tau}}\int^{t}_{t+\theta}m^{T}(s,x)R_{3}m(s,x)\mathrm{d}s\mathrm{d}\theta\mathrm{d}x
$$
and
$$
\int_{\Omega}\int^{0}_{-\overline{\tau}}\int^{0}_{\theta}\int^{t}_{t+\lambda}\hspace{-2mm}\frac{\partial m^{T}(s,x)}{\partial s}M_{1}\frac{\partial m(s,x)}{\partial s}\mathrm{d}s\mathrm{d}\lambda\mathrm{d}\theta\mathrm{d}x
$$
into Lyapunov-Krasovskii functional and employ Wirt\-ing\-er-based integral inequality (instead of Jensen's inequality) to estimate the derivative of the second one,
which will get more accurate result.
  \item The so-called convex combination approach and RCC approach are employed simultaneously, which will improve the precision of estimation to the concentrations of mRNA and protein.
  \item The coefficients of some items in $\varsigma(t,x)$, like $\frac{1}{\tau(t)}$ and $\frac{1}{\overline{\tau}-\tau(t)}$, play a very important role in simplification of the LMI condition (\ref{equation-11}).
  \item We use $\varphi(t,s,x)$ and $\psi(t,s,x)$ instead of $m(s,x)$ and $p(s,x)$ in $V_2(t,m,p)$, respectively.
  This will highly maintain consistent with $V_3(t,m,p)$.
\end{enumerate}
\end{remark}

\begin{remark}
The approach proposed in this paper can easily be applied to establish a delay-dependent and delay-rate-dependent asymptotic stability criterion for GRN (\ref{equation-1}). Due to Remark \ref{15sep07-remark-1} above, the criterion is certainly less conservative than ones in \cite{NCA-2011-507,001,Han.Zhang.Wang-CSSP}.
\end{remark}

\section{ Illustrative examples}\label{2015}

In this section, two numerical examples are provided to demonstrate the effectiveness and applicability of the proposed state observer.

\begin{example}\label{1}
Consider GRN (\ref{equation-4}) with measurements (\ref{equation-6}), the deterministic parameters are given as:
$$A=\mathrm{diag}(0.2,1.1,1.2),\ B=\mathrm{diag}(1.0,0.4,0.7),$$
$$C=\mathrm{diag}(0.3,0.7,1.3),\ L_{1}=L_{2}=L_{3}=1,$$
$$
W=\left [\begin{array}{ccc}
0&0&-0.5\\-0.5&0&0\\0&-0.5&0
\end{array}\right],
$$
$$D_{1}=D_{2}=D_{3}=\mathrm{diag}(0.1,0.1,0.1),$$
$$D^{*}_{1}=D^{*}_{2}=D^{*}_{3}=\mathrm{diag}(0.2,0.2,0.2),$$
$$
M=\left [\begin{array}{ccc}
0.5& -0.6 &0\\0.3 &0.8&-0.2
\end{array}\right],
$$
$$
N=\left [\begin{array}{ccc}
0.7 &-0.25 &0.3\\0.4 &0.2 &-0.3
\end{array}\right].
$$
Here the regulation function is taken as $f(x) =\frac{x^{2}}{1+x^{2}}$. One can get (\ref{equation-5}) holds when $K =0.65I$.
When $\overline{\tau} = \overline{\sigma} = 3$ and $\mu_{1}=\mu_{2}=2$, by using the MATLAB YALMIP Toolbox, one can see that the LMIs
given in Theorem \ref{theorem1-1} are feasible with the following feasible
solution matrices. To save space, we only list some of
the feasible solution matrices as follows:
$$
P_{1}=\mathrm{diag}( 57.6506,44.1104,50.5774),
$$
$$
P_{2}=\mathrm{diag}(25.7909,39.4682,32.9357),
$$
$$
Q_{1}=\left[\begin{array}{ccc}1.3165&-0.0077&-0.0077\\
   -0.0077&2.2460&0.0907\\
   -0.0077 &0.0907&1.5577
\end{array}\right],
$$
$$
Q_{5}=\left[\begin{array}{ccc}  2.8401&0.1585&0.0103\\
    0.1585&5.6978&0.0748\\
    0.0103&0.0748&3.3839
\end{array}\right],
$$
$$
R_{1}=\left[\begin{array}{ccc}  5.8721&0.0099&0.0606\\
    0.0099&3.7715&-0.1517\\
    0.0606&-0.1517&4.2376
\end{array}\right],
$$
$$
R_{2}=\left[\begin{array}{ccc}    2.0575&-0.0035&0.0052\\
   -0.0035&3.6837&-0.1101\\
    0.0052&-0.1101&2.8021
\end{array}\right],
$$
$$
M_{1}=\left[\begin{array}{ccc}     1.6422& 0.0046&0.0301\\
    0.0046&1.4934&-0.0324\\
    0.0301&-0.0324&1.2873
\end{array}\right],
$$
$$
W_{1}=\left[\begin{array}{cc}     34.7528&25.0882\\
    9.8144&-15.2841\\
    4.3837&9.3021
\end{array}\right],
$$
$$
W_{2}=\left[\begin{array}{cc}      16.3285&18.6193\\
    5.3137&-9.3968\\
  -12.5043&22.4536
\end{array}\right].
$$
Moreover, we can get the corresponding observer gain matrices as follows:
$$
K_{1}=P^{-1}_{1}W_{1}=\left[\begin{array}{cc}     0.6028&0.4352\\
    0.2225&-0.3465\\
    0.0867&0.1839
\end{array}\right],
$$
$$
K_{2}=P^{-1}_{2}W_{2}=\left[\begin{array}{cc}     0.6331&0.7219\\
    0.1346&-0.2381\\
   -0.3797&0.6817
\end{array}\right].
$$
\end{example}

\begin{example}\label{2}
When $l=n=1$, GRN (\ref{equation-4}) is simplified into
 \begin{equation}\label{32}
\left\{\begin{array}{rl}
\frac{\partial \bar{m}(t,x)}{\partial t}=&\frac{\partial}{\partial x}\left(D_{1}\frac{\partial \bar{m}(t,x)}{\partial x}\right)-A\bar{m}(t,x)\\
&+Wf(\bar{p}(t-\sigma(t),x)),\\
\frac{\partial \bar{p}(t,x)}{\partial t}=&\frac{\partial}{\partial x}\left(D^{*}_{1}\frac{\partial \bar{p}(t,x)}{\partial x}\right)-C\bar{p}(t,x)\\
&+B\bar{m}(t-\tau(t),x).
\end{array}\right.
\end{equation}

We choose the values of parameters in (\ref{32}) are as follows,
$$A=0.2, \ B=1, \ C=0.3, \ L_{1}=1, $$ $$W=-0.5, \ D_{1}=0.1, \ D^{*}_{1}=0.2,$$
$$
M=0,\ \
N=0.7 .
$$
When $\mu_{1}=\mu_{2}=2, K=0.65$ and $\overline{\tau}=\overline{\sigma}=1,$ for Dirichlet boundary conditions, by using the MATLAB YALMIP Toolbox to solve the
LMIs given in Theorem \ref{theorem1-1}, we obtain the following feasible solution matrices. To save space, we only list some of
the feasible matrices as follows:
$$P_{1}=1.8102, \ \ Q_{1}=0.0120, $$
 $$ R_{1}=0.2847, \ \ M_{1}=0.1556,$$
 $$
W_{1}=0, \ \ \
W_{2}=1.1478.
$$
Moreover, we can get the corresponding observer gain matrices as follows:
$$
K_{1}=P^{-1}_{1}W_{1}=0,
$$
$$
K_{2}=P^{-1}_{2}W_{2}=1.5571.
$$
Further, when $\sigma(t)=\tau(t)=1$, the state responses of GRN (\ref{32}), observer (\ref{equation-7}) and the corresponding error system are given in Figures \ref{x}--\ref{yyy}.
\begin{figure}
  \centering
  \includegraphics[width=0.35\textwidth]{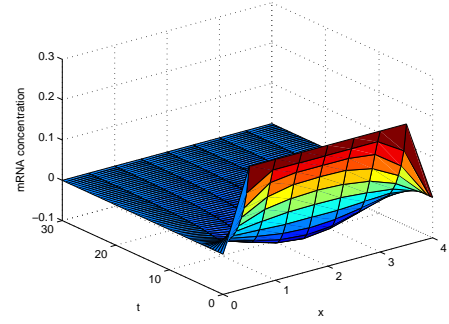}\\
  \caption{The real trajectory of mRNA ($\bar{m}(t,x)$) }\label{x}
\end{figure}
\begin{figure}
  \centering
  \includegraphics[width=0.35\textwidth]{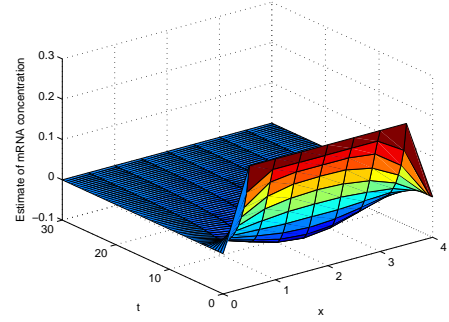}\\
  \caption{The estimated trajectory of mRNA ($\hat{m}(t,x)$)}\label{y}
\end{figure}
\begin{figure}
  \centering
  \includegraphics[width=0.35\textwidth]{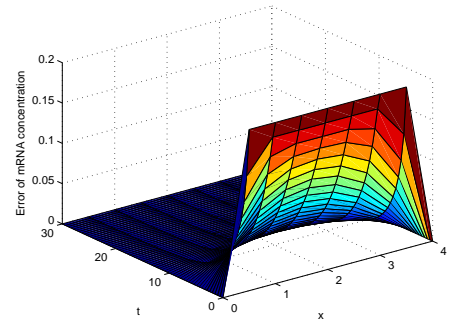}\\
  \caption{The estimation error of mRNA ($\bar{m}(t,x)-\hat{m}(t,x)$)}\label{xx}
\end{figure}
\begin{figure}
  \centering
  \includegraphics[width=0.35\textwidth]{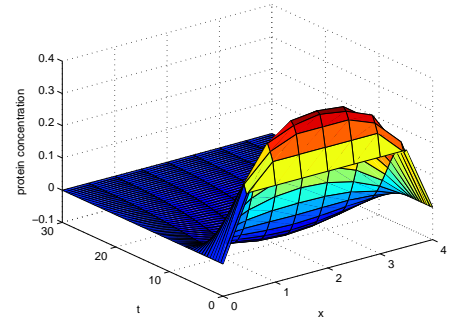}\\
  \caption{The real trajectory of protein ($\bar{p}(t,x)$)}\label{yy}
\end{figure}
\begin{figure}
  \centering
  \includegraphics[width=0.35\textwidth]{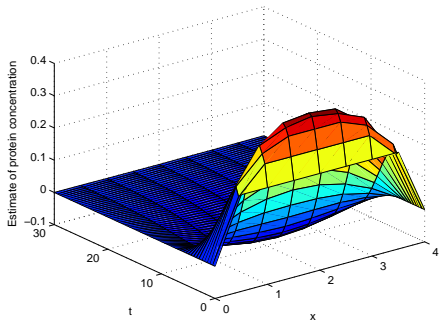}\\
  \caption{The estimated trajectory of protein ($\hat{p}(t,x)$) }\label{xxx}
\end{figure}
\begin{figure}
  \centering
  \includegraphics[width=0.35\textwidth]{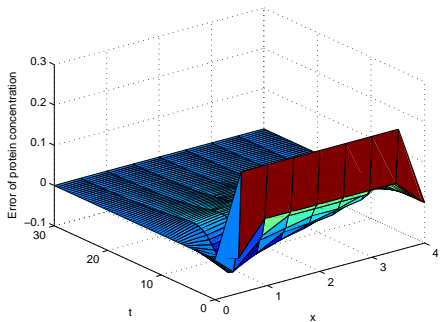}\\
  \caption{The estimation error of protein ($\bar{p}(t,x)-\hat{p}(t,x)$)}\label{yyy}
\end{figure}

\end{example}

\section{Conclusions}\label{2016}
In this paper, the state estimation problem for a class of GRNs with time-varying delays and reaction-diffusion terms are studied.
 An state observer is designed to estimate the gene states through available sensor measurements, and guarantee that the
error system is asymptotically stable. By introducing new integral terms in a novel Lyapunov--Krasovskii functional and employing the so-called Wirtinger-based integral inequality, Wirtinger's inequality, Green's second identity and, convex combination approach, RCC approach, a sufficient condition guaranteeing the existence of
state observers is established in terms of LMIs. The concrete expression of the desired state observer has been
presented in Theorem \ref{theorem1-1}. Finally, two numerical examples are given to illustrate the effectiveness of the theoretical results.







\end{document}